\documentclass[12pt]{amsart}

\usepackage{hyperref}

\newtheorem{theorem}{Theorem}[section]

\newtheorem*{claim}{Claim}
\newtheorem{lemma}[theorem]{Lemma}

\newtheorem{conjecture}[theorem]{Conjecture}

\newcommand{\bF}{\mathbb F}

\newcommand{\cB}{\mathcal{B}}

\newcommand{\cI}{\mathcal{I}}

\newcommand{\del}{\setminus}
\newcommand{\con}{/}

\DeclareMathOperator{\loops}{loops}

\DeclareMathOperator{\cl}{cl}

\begin{document}

\sloppy

\title[Quasi-graphic matroids]{Quasi-graphic matroids}

\author{Jim Geelen}
\address{Department of Combinatorics and Optimization,
University of Waterloo, Waterloo, Canada} 
\thanks{This research was partially supported by grants from the
Office of Naval Research [N00014-10-1-0851], NSERC [203110-2011],
and the Marsden Fund of New Zealand.}

\author{Bert Gerards}
\address{Centrum Wiskunde \& Informatica,
Amsterdam, The Netherlands}

\author{Geoff Whittle}
\address{School of Mathematics, Statistics and Operations Research,
Victoria University of Wellington, New Zealand}

\date{\today}

\begin{abstract}
Frame matroids and lifted-graphic matroids are two
interesting generalizations of graphic matroids.
Here we introduce a new generalization, {\em quasi-graphic matroids},
that unifies these two existing classes. Unlike frame matroids
and lifted-graphic matroids, it is easy to certify that
a $3$-connected matroid is quasi-graphic. The main result
is that every $3$-connected representable quasi-graphic matroid is
either a lifted-graphic matroid or a frame matroid.
\end{abstract}

\maketitle

\section{Introduction}
Let $G$ be a graph and let $M$ be a matroid.
For a vertex $v$ of $G$ we let $\loops_G(v)$
denote the set of loop-edges of $G$ at the vertex $v$.
We say that $G$ is a {\em framework} for $M$ if
\begin{itemize}
\item[(1)] $E(G)=E(M)$,
\item[(2)] $r_M(E(H))\le |V(H)|$ for each component $H$ of $G$, 
\item[(3)] for each vertex $v$ of $G$ we have
$\cl_M(E(G-v))\subseteq E(G-v)\cup\loops_G(v)$, and
\item[(4)] for each circuit $C$ of $M$, the subgraph
$G[C]$ has at most two components.
\end{itemize}

An earlier version of this paper had a serious flaw that was
pointed out to us by Daryl Funk.  In order to overcome that
issue we added condition $(4)$ to the definition.
 
This definition is motivated by the 
following theorem that follows from a result
due to Seymour~[\ref{seymour}].
\begin{theorem}\label{graphic}
Let $G$ be a graph with $c$ components and let $M$ be a matroid.
Then $M$ is the cycle matroid of $G$ if and only if
$G$ is a framework for $M$ and $r(M) \le |V(G)|-c$.
\end{theorem}
 
We will call a matroid {\em quasi-graphic}
if it has a framework. Next we will consider two 
classes of quasi-graphic matroids; namely ``lifted-graphic
matroids" and ``frame matroids".

We say that a matroid $M$ is a {\em lifted-graphic matroid} 
if there is a matroid $M'$ and an element $e\in E(M')$ such that
$M'\del e = M$ and $M'\con e$ is graphic. 
The following result is proved in Section~\ref{sec:lifts}.
\begin{theorem}\label{lifts}
Every lifted-graphic matroid is quasi-graphic.
\end{theorem}

A {\em framed matroid} is a pair $(M,V)$ such that
$M$ is a matroid, $V$ is a basis of $M$, and
each element of $M$ is spanned by a subset
of $V$ with at most two elements.
A matroid $M$ is a {\em frame matroid} if there 
is a framed matroid $(M',V)$ such that $M=M'\del V$.
The following result is proved in Section~\ref{sec:frames}.
\begin{theorem}\label{frame}
Every frame matroid is quasi-graphic.
\end{theorem}

Our main result is that for matroids that are both $3$-connected
and representable, there are no quasi-graphic matroids
other than those described above.
\begin{theorem}\label{main}
Let $M$ be a $3$-connected representable matroid.
If $M$ is quasi-graphic, then $M$ is either
a frame matroid or a lifted-graphic matroid.
\end{theorem}

The representability condition in Theorem~\ref{main}
is necessary; the V\' amos matroid, for example, is
quasi-graphic but it is neither
a frame matroid nor a lifted-graphic matroid.
However, for frameworks with loop-edges, we do not require
representability.
\begin{theorem}\label{sufficient}
Let $G$ be a framework for a $3$-connected matroid $M$.
If $G$ has a loop-edge, then $M$ is either a frame matroid or 
a lifted-graphic matroid.
\end{theorem}
Our proof of Theorem~\ref{sufficient}
uses results of Zaslavsky~[\ref{zaslavsky}] who characterized
frame matroids and lifted-graphic matroids using
``biased graphs"; we review those results in Sections~\ref{sec:lifts}
and~\ref{sec:frames}.

Given a graph $G$ and a matroid $M$, via its rank oracle, one
can efficiently check the conditions (1), (2), and (3). However,
since $M$ may have exponentially many circuits, it is not clear
how one might efficiently check (4). The following result
shows that one can easily certify that a $3$-connected matroid
is quasi-graphic.
\begin{theorem}\label{certify}
A $3$-connected matroid $M$ is quasi-graphic if 
and only if there exists 
a graph $G$ such that
\begin{itemize}
\item[(i)] $E(G)=E(M)$,
\item[(ii)] $G$ is connected,
\item[(iii)] $r(M)\le |V(G)|$, and
\item[(iv)] for each vertex $v$ of $G$ we have
$\cl_M(E(G-v))\subseteq E(G-v)\cup\loops_G(v)$.
\end{itemize}
\end{theorem}

We conjecture that the problem of recognizing $3$-connected
quasi-graphic matroids is tractable.
\begin{conjecture}\label{bigconj}
There is a polynomial-time algorithm that given a 
$3$-connected matroid $M$, via its rank-oracle,
decides whether or not $M$ is quasi-graphic.
\end{conjecture}

This contrasts with results of Rong and Whittle~[\ref{rw}]
who prove that there is no efficient algorithm 
for either frame matroid recognition or 
lifted-graphic matroid recognition.
In fact, Rong and Whittle's results show that
to certify that a matroid is either a frame matroid
or a lifted graphic matroid requires an exponential number
of rank evaluations in the worst case; those results were 
conjectured in an earlier version of this paper.

We will use the notation and terminology of Oxley~[\ref{oxley}],
except we denote $|E(M)|$ by $|M|$ and we define a graph $G$
to be {\em $k$-connected} when $G-X$ is connected for each set
$X\subseteq V(G)$ with $|X|<k$ (we do not require that $|V(G)|>k$);
moreover, we consider that the graph with no vertices is connected.

\section{Weak frameworks and minors}

For the aid of readers familiar with the earlier
flawed version of this paper, we will keep the following two 
sections essentially unchanged except for correcting 
the transgression that was the root cause for the mistakes.
To this end, we need to introduce a version of frameworks 
without the condition (4).
 
Let $M$ be a matroid and $G$ be a graph.
We call $G$ a {\em weak framework}  for $M$ if 
\begin{itemize}
\item[(1)] $E(G)=E(M)$,
\item[(2)] $r_M(E(H))\le |V(H)|$ for each component $H$ of $G$, and
\item[(3)] for each vertex $v$ of $G$ we have
$\cl_M(E(G-v))\subseteq E(G-v)\cup\loops_G(v)$.
\end{itemize}
Note that a graph $G$ satisfies conditions $(i)-(iv)$ of
Theorem~\ref{certify} if and only if it is a connected weak framework.
In this section we will 
see that weak frameworks behave nicely under minors.

\begin{lemma}\label{component}
Let $G$ be a weak framework for a matroid $M$.
If $H$ is a component of $G$, then 
$H$ is a weak framework for $M|E(H)$.
\end{lemma}

\begin{proof}
Note that conditions $(1)$ and $(2)$ are immediate.
Condition $(3)$ follows from the fact that
for each flat $F$ of $M$, the set $F\cap E(H)$ 
is a flat of $M|E(H)$.
\end{proof}

The following result is very easy, but it is used repeatedly.
\begin{lemma}\label{deletevertex}
Let $G$ be a weak framework for a matroid $M$.
If $v$ is a vertex of $G$ that is incident with 
at least one non-loop-edge, then $r_M(E(G-v))<r(M)$.
Moreover, if $v$ has degree one, then 
$r_M(E(G-v))=r(M)-1$.
\end{lemma}

\begin{proof}
This follows directly from $(3)$.
\end{proof}

\begin{lemma}\label{subgraph}
Let $G$ be a connected weak framework for a matroid $M$ and let 
$H$ be a non-empty subgraph of $G$. Then
$|V(H)|-r(M|E(H))\ge |V(G)| - r(M)$.
\end{lemma}

\begin{proof}
We can extend $H$ to a spanning subgraph $H^+$ of $G$,
adding one vertex and one edge at a time,  with
$|E(H^+)|-|E(H)|= |V(G)|-|V(H)|$. Clearly
$|V(H^+)|-r(E(H^+))\ge |V(G)| - r(M)$.
If $H\neq H^+$, then there is a vertex $v\in V(H^+)-V(H)$
that has degree one in $H^+$. By Lemma~\ref{deletevertex},
$r(E(H^+-v)) = r(E(H))-1$ and, hence,
$|V(H^+-v)|-r(E(H^+-v))\ge |V(G)| - r(M)$.
Now we obtain the result by repeatedly deleting vertices
in $V(H^+)-V(H)$ in this way.
\end{proof}

If $X$ is a set of edges in a graph $G$, then $G[X]$ is the subgraph
of $G$ with edge-set $X$ and with no isolated vertices;
moreover, we will denote $V(G[X])$ by $V(X)$.
\begin{lemma}\label{restriction}
If $G$ is a weak framework for a matroid $M$ and 
$X\subseteq E(M)$, then $G[X]$ is a weak framework for $M|X$.
\end{lemma}

\begin{proof}
Condition $(1)$ is clearly satisfied.
Condition $(2)$ follows from Lemmas~\ref{component}
and~\ref{subgraph}.
Condition $(3)$ follows from the fact that
for each flat $F$ of $M$, the set $F\cap E(H)$ 
is a flat of $M|E(H)$.
Thus $G[X]$ is a weak framework for $M|X$.
\end{proof}

The following two results give sufficient conditions for
independence and dependence, respectively, for a set in
a quasi-graphic matroid.
\begin{lemma}\label{forest}
Let $G$ be a weak framework for a matroid $M$.
If $F$ is a forest of $G$, then $E(F)$ is an independent
set of $M$.
\end{lemma}

\begin{proof}
We may assume that $E(F)$ is non-empty and, hence, that
$F$ has a degree-one vertex $v$. By Lemma~\ref{deletevertex},
$r_M(E(F)) = r_M(E(F-v))+1$. Now the result
follows inductively.
\end{proof}

\begin{lemma}\label{dependentset}
Let $G$ be a weak framework for a matroid $M$. 
If $H$ is a subgraph of $G$ and
$|E(H)|>|V(H)|$,
then $E(H)$ is a dependent set of $M$.
\end{lemma}

\begin{proof}
By Lemma~\ref{restriction} and $(2)$, we have
$r_M(E(H))\le |V(H)|$. So, if $|E(H)|>|V(H)|$,
then $E(H)$ is a dependent set of $M$.
\end{proof}

We can now prove Theorem~\ref{graphic}.
The ``only if" direction is routine and left to the reader.
For the ``if" direction we prove the following stronger result
in which we have replaced ``framework" with ``weak framework".
This result is tantamount to the main theorem of~[\ref{seymour}],
but we include the proof since it is short and the
result is central to this paper.
\begin{theorem}\label{strong-graphic}
Let $G$ be a graph with $c$ components and let $M$ be a matroid.
If $G$ is a weak framework for $M$ and $r(M) \le |V(G)|-c$,
then $M$ is the cycle matroid of $G$.
\end{theorem}

\begin{proof}
By Lemma~\ref{forest} and the fact that
$r(M)\le |V(G)|-c$, we have $r(E(H)) = |V(H)|-1$ for each component
$H$ of $G$. Hence we may assume that $G$ is connected.
By Lemma~\ref{forest}, the edge-set of each forest of $G$ is independent
in $M$. Therefore, it suffices to prove, for each cycle
$C$ of $G$, that $E(C)$ is dependent in $M$.
By Lemma~\ref{subgraph},
$|V(C)|-r(E(C)) \ge |V(G)| - r(E(G)) = 1$.
So $r(E(C)) < |V(C)|=|E(C)|$ and, hence, $E(C)$ is 
dependent as required.
\end{proof}

To consider the effect of contraction on weak frameworks,
we consider two cases depending on whether or not we are
contracting a loop-edge.
\begin{lemma}\label{contract}
Let $G$ be a  weak framework for a matroid $M$ and let 
$e$ be a non-loop-edge of $G$.
Then $G\con e$ is a weak framework for $M\con e$.
\end{lemma}

\begin{proof}
Conditions $(1)$ and $(2)$ are clearly satisfied.
Let $u$ and $v$ be the ends of $e$ in $G$, and let $f$
be an edge of $G$ that is incident with $u$ but not with 
$v$. To prove $(3)$ it suffices to prove that
that there exists a cocircuit $C$ in $M$
such that $f\in C$, $e\not\in C$, and $C$ contains
only edges incident with either $u$ or $v$.

By $(3)$, there exist cocircuits $C_e$ and $C_f$ such that
$e\in C_e$, that $C_e$ contains only edges incident with $v$, that
$f\in C_f$, and that $C_f$ contains only edges incident with $u$.
We may assume that $e\in C_f$ since otherwise we could take $C=C_f$.
Since $f$ is not incident with $v$, we have $f\not\in C_e$.
Then, by the strong circuit exchange axiom, there is a 
cocircuit $C$ of $M$ with $f\in C\subseteq (C_1\cup C_2)-\{e\}$,
as required.
\end{proof}

Let $e$ be a loop-edge of a graph $G$ and let $v$ be the vertex 
incident with $v$. We denote by $G\circ e$ the graph obtained from $G-v$ by first,
for each non-loop edge $f=vw$ incident with $v$
adding $f$ as a loop-edge at $w$, and then
for each loop-edge $f$ of $G-e$ at $v$ adding
$f$ as a loop-edge on an arbitrary vertex.
The graph $G\circ e$ is well-defined unless there
are multiple loop-edges at $v$. 
\begin{lemma}\label{contractloop}
Let $G$ be a weak framework for $M$ and let $e$
be a loop-edge of $G$.
If $e$ is not a loop of $M$, then
$G\circ e$ is a weak framework for $M\con e$.
\end{lemma}

\begin{proof}
Let $v$ be the vertex incident with $e$.
Condition $(1)$ is clearly satisfied and
condition $(2)$ is also routine.
By Lemma~\ref{restriction} and (2), we have $r_M(\loops_G(v)) = 1$,
so each element of $\loops_G(v)-\{e\}$ is a loop in $M\con e$.
Each vertex $w\in V(G)-\{v\}$ is incident with the same edges
in $G$ as it is in $H$ except for the elements in $\loops_G(v)$.
Moreover, $\cl_{M}(E(G-w)) = \cl_{M\con e}(E(H-w))\cup \{e\}$.
Therefore $(3)$ follows.
\end{proof}

\section{Balanced cycles}

Let $G$ be a weak framework for a matroid $M$ and let $C$ be a cycle of $G$.
By Lemmas~\ref{subgraph} and~\ref{forest},
$E(C)$ is either independent in $M$ or $E(C)$ is a circuit in $M$.
If $E(C)$ is a circuit of $M$, then we say
that $C$ is a {\em balanced} cycle of $(M,G)$;
when the matroid $M$ is clear from the context, we will say that
$C$ is a {\em balanced} cycle of $G$. We recall that $M(G)$ 
denotes the cycle matroid of a graph $G$.
\begin{lemma}\label{balanced}
Let $G$ be a weak framework for a matroid $M$.
Then $M=M(G)$ if
and only if each cycle of $G$ is balanced.
\end{lemma}

\begin{proof}
If $M=M(G)$, then each cycle of $G$ is balanced.
Conversely, suppose that each cycle of $G$ is balanced.
Let $F$ be a maximal forest in $G$.
Since each cycle is balanced, $E(F)$ is a basis of $M$.
Then, by Theorem~\ref{graphic}, $M=M(G)$.
\end{proof}

A  {\em theta} is a 
$2$-connected graph that has exactly two vertices of degree
$3$ and all other vertices have degree $2$. 
Observe that there are exactly three cycles in a theta.

\begin{lemma}\label{theta}
Let $G$ be a weak framework for a matroid  $M$ and let $H$
be a theta-subgraph of $G$. If two of the cycles in $H$ are
balanced, then so too is the third.
\end{lemma}

\begin{proof}
If there are two balanced cycles in $H$ then
$r_M(E(H)) \le |E(H)|-2= |V(H)|-1$.
So, by Theorem~\ref{graphic}, $M|E(H)=M(H)$ and, by 
Lemma~\ref{balanced}, all cycles of $H$ are balanced.
\end{proof}

The following result describes the circuits of a matroid
in terms of its weak framework; first we will give an
unusual example to demonstrate one of the outcomes.
If $M$ consists of a single circuit and $G$ is a graph
with $E(G)=E(M)$ whose components are cycles,
then $G$ is a weak framework for $M$.
\begin{lemma}\label{circuit}
Let $G$ be a weak framework for a matroid $M$.
If $C$ is a circuit of $M$, then either
\begin{itemize}
\item $G[C]$ is a balanced cycle,
\item $G[C]$ is a connected graph with minimum degree at least two,
$|C| = |V(C)| +1$, and
$G[C]$ has no balanced cycles, or
\item $G[C]$ is a collection of vertex-disjoint non-balanced cycles.
\end{itemize} 
\end{lemma}

\begin{proof}
We may assume that $G[C]$ is not a balanced cycle, and, hence,
that $G[C]$ contains no balanced cycle.
Next suppose that $|C|\ge |V(C)| +1$.
By Lemma~\ref{dependentset}, $C$ is minimal with this property.
Hence $G[C]$ is connected, the minimum degree of $G[C]$ is two,
and $|C|= |V(C)| +1$. 
Now suppose that $|C|\le |V(C)|$ and
consider a component $H$ of $G[C]$; it suffices to show that
$H$ is a cycle.
By Lemma~\ref{dependentset} and the argument above, we may
assume that $|E(H)|\le |V(H)|$.
If $H$ is not a cycle there is a degree-one vertex $v$ of $H$. 
Moreover, the edge $e$ that is incident with $v$ is not a loop-edge.
Then, by $(3)$, the element $e$ is a coloop of $M|C$, which contradicts
the fact that $C$ is a circuit.
\end{proof}

For a set $X$ of elements in a matroid $M$ we let 
$$ \lambda_M(X) = r_M(X) + r_M(E(M)-X) - r(M).$$
\begin{lemma}\label{conn}
Let $G$ be a weak framework for $M$. If $H$ is a component
of $G$, then $\lambda_M(E(H))\le 1$.
\end{lemma} 

\begin{proof}
By Lemma~\ref{deletevertex},
$r(E(M)-E(H)) \le r(M) - (|V(H)|-1)$.
Hence $\lambda_M(E(H)) = r_M(E(H)) + r_M(E(M)-E(H)) - r(M)
\le |V(H)| + (r(M) - (|V(H)|-1)) -r(M) = 1.$
\end{proof}

A {\em loop-component} of a graph $G$ is a component
consisting of exactly one vertex and exactly one edge.
The mistake in the earlier version of this paper
was that the third outcome of the following 
lemma was overlooked. 
\begin{lemma}\label{conn2}
If $G$ is a weak framework for a $3$-connected matroid $M$
with $|M|\ge 4$ and $G$ has no isolated vertices,
then either
\begin{itemize}
\item[(a)] $G$ is connected, 
\item[(b)] $G$ has exactly two components one of which is a 
loop-component, or
\item[(c)] every component of $G$ is a loop-component. 
\end{itemize}
\end{lemma}

\begin{proof}
It follows from $(3)$  and the fact that 
$M$ has no coloops that each component 
of $G$ is either a loop-component or
has at least two edges.
We may assume that $G$ has a non-loop-component, say $H$,
and we may further assume that $|E(G)-E(H)|\ge 2$.
Then, by Lemma~\ref{conn}, 
$(E(H), E(G)-E(H))$ is a $2$-separation. However
this contradicts the fact that  
$M$ is $3$-connected.
\end{proof}

That oversight turns out to be catastrophic. Let 
$M$ be a matroid and let $G$ be a graph, with 
$E(G)=E(M)$, whose components are all loop-components.
Then $G$ is a weak framework for $M$.
So all matroids admit weak frameworks!

To conclude this section we consider additional properties of
the weak frameworks that
satisfying outcomes~(a) and~(b) of Lemma~\ref{conn2}.
We will start by showing that, in case~(b),
$M$ is a coextension of a graphic matroid.

\begin{lemma}\label{xx}
Let $G$ be a weak framework, without isolated vertices,
for a connected matroid $M$.
If $G$ has exactly two components and $e$ is an edge in a 
loop-component of $G$, then $M\con e = M(G\del e)$.
Moreover $M$ has a connected weak framework.
\end{lemma}

\begin{proof}
Let $H$ denote the component of $G$ that does not contain $e$.
By Lemma~\ref{contractloop},
$H$ is a connected weak framework for $M\con e$.
By Lemma~\ref{component}, we have $r_M(E(H)) \le |V(H)|$.
Since $M$ is connected, $e\in \cl_M(E(H))$.
Therefore $r(M\con e) \le |V(H)|-1$.
Then, by Theorem~\ref{strong-graphic}, we have
$M\con e = M(H)$.

Let $v$ be the vertex of $G$ incident with $e$ and let 
$w\in V(G)\del\{v\}$.  Construct a graph
$G^+$ from $G$ by adding a new edge $f$ with ends $v$ and $w$
and construct a new matroid $M^+$ by adding 
$f$ as a coloop to $M$. Note that
$G^+$ is a weak framework for $M^+$ and hence, by
Lemma~\ref{contract}.
$G^+\con f$ is a weak framework for $M^+\con f$.
Moreover, as $f$ is a coloop of $M^+$, we have $M^+\con f=M$.
Therefore $G^+\con e$ is a connected weak framework for $M$.
\end{proof}

The following result
shows that connected weak frameworks are in fact
frameworks; this implies the ``if" direction
of Theorem~\ref{certify}.
\begin{lemma}\label{circuit3}
Let $M$ be a matroid with at least four elements
and let $G$ be a connected weak framework for $M$.
If $C$ is a circuit of $M$, then $G[C]$ has at most 
two components.
\end{lemma}

\begin{proof}
Suppose that $G[C]$ has more than two components.
By Lemma~\ref{circuit}, each component
of $G[C]$ is a non-balanced cycle. 
Let $P$ be a shortest path connecting two components
of $G[C]$; let these components be $C_1$ and $C_2$.
Since $C$ is a circuit, $E(C_1\cup C_2)$ is
independent. By Lemmas~\ref{dependentset} and~\ref{circuit},
$E(C_1\cup C_2\cup P)$ is a circuit of $M$. 

Let $e\in E(P)$ and $f\in E(C_1)$. By the strong
exchange property for circuits, there is a circuit $C'$
of $G$ with $e\in C' \subseteq (C\cup E(P))-\{f\}$.
However this is inconsistent with the outcomes
of Lemma~\ref{circuit}. 
\end{proof}

Finally  we show that every connected weak framework
for a $3$-connected matroid is necessarily $2$-connected.
\begin{lemma}\label{conn4}
Let $M$ be a $3$-connected matroid with $|M|\ge 4$.
If $G$ is a connected weak framework for $M$, then 
$G$ is $2$-connected.
\end{lemma}

\begin{proof}
Suppose otherwise. Then there is a pair
$(H_1,H_2)$ of subgraphs of $G$ such that
$G=H_1\cup H_2$, $|V(H_1)\cap V(H_2)|=1$, and $|V(H_1)|,|V(H_2)|\ge 2$.
Note that $H_1$ and $H_2$ are both connected.
Now $M(G)$ is not $3$-connected, so, by Theorem~\ref{graphic},
$r(M) = |V(G)|$. Therefore
$\lambda_M(E(H_1)) \le |V(H_1)| + |V(H_2)| - |V(G)| = 1$.
Since $M$ is $3$-connected either
$|E(H_1)|\le 1$ or $|E(H_2)|\le 1$; we may assume that
$|E(H_1)|=1$. Let $e\in E(H_1)$. 
Since $H_1$ is connected and $|V(H_1)|\ge 2$, 
the edge $e$ is not a loop-edge.
Therefore, by $(3)$, $e$ is a coloop of $M$.
This contradicts the fact that $M$ is $3$-connected.
\end{proof}

\section{Quasi-graphic matroids}

The following result shows that the class of quasi-graphic matroids
is closed under taking minors.
\begin{lemma}\label{minors}
Let $G$ be a framework for a matroid $M$.
For each $e\in E(M)$,
\begin{itemize}
\item $G-e$ is a framework for $M\del e$,
\item if $e$ is not a loop-edge of $G$, then $G\con e$
is a framework for $M\con e$, and
\item if $e$ is a loop-edge of $G$ and $e$ is not a loop of $M$,
then $G\circ e$ is a framework for $M\con e$.
\end{itemize}
\end{lemma}

\begin{proof}
By Lemma~\ref{restriction}, $G-e$ is a 
weak framework for $M\del e$.
Moreover, $(4)$ is clearly preserved under 
deletion, so $G-e$ is a 
framework for $M\del e$.

Suppose that $e$ is a non-loop-edge of $G$.
By Lemma~\ref{contract}, $G\con e$ is a 
weak framework for $M\con e$.
Consider a circuit $C$ of $M\con e$.
Either $C$ or $C\cup \{e\}$ is a circuit of $M$.
Let $C'\in\{C,C\cup\{e\}\}$ be a circuit of $M$.
Then $G[C']$ has at
most two components, and, hence, $G\con e[C]$
has at most two components. 
So  $G\con e$ is a framework for $M\con e$.

Finally suppose that $e$ is a loop-edge of $G$ and that
$e$ is not a loop of $M$.
By Lemma~\ref{contractloop}, $G\circ e$ is a 
weak framework for $M\con e$.
Consider a circuit $C$ of $M\con e$.
Either $C$ or $C\cup \{e\}$ is a circuit of $M$.
Let $C'\in\{C,C\cup\{e\}\}$ be a circuit of $M$.
Then $G[C']$ has at
most two components and, by Lemma~\ref{circuit},
if $G[C']$ has two components then each of the components
is $2$-connected. Thus $G\circ e[C]$
has at most two components. 
So  $G\circ e$ is a framework for $M\con e$.
\end{proof}

The following result is a strengthening
of Lemma~\ref{conn2} for frameworks.
\begin{lemma}\label{conn3}
If $G$ is a framework for a $3$-connected matroid $M$
with $|M|\ge 4$ and $G$ has no isolated vertices,
then either
\begin{itemize}
\item[(a)] $G$ is connected, or 
\item[(b)] $G$ has exactly two components one of which is a 
loop-component.
\end{itemize}
Moreover $M$ has a connected framework.
\end{lemma}

\begin{proof}
By Lemma~\ref{conn2}, if $G$ does not satisfy 
(a) or (b), then each component of $G$ is a
loop-component. Since $M$ is $3$-connected, and therefore
simple, $M$ has a circuit of length at least $3$.
However any such circuit violates (4).
Hence $G$ indeed satisfies (a) or (b).

If $G$ is itself not connected, then $G$ satisfies (b).
Then, by Lemma~\ref{xx}, $M$ has a connected
weak framework $G'$. By Lemma~\ref{circuit3},
$G'$ is a connected framework for $M$. 
\end{proof}

Note that Theorem~\ref{certify}
follows directly from Lemmas~\ref{conn3} and~\ref{circuit3}.

\section{Frame matroids}
\label{sec:frames}

A {\em simple framed matroid} is a framed matroid
$(M,V)$ with $M$ simple.
The {\em support graph} of a simple framed matroid $(M,V)$ is the 
graph $G=(V,E(M))$ such that, for  each $v\in V$,
the edge $v$ is a loop-edge at the vertex $v$, and,
for each $e\in E(M)-V$,
the edge $e$ has ends $u$ and $v$ where
$\{e,u,v\}$ is the unique circuit of $M$ in $V\cup\{e\}$.

\begin{lemma}\label{frame1}
If $G$ is the support graph of a simple framed matroid
$(M,V)$, then $G$ is a framework for $M$ and for each circuit
$C$ of $M$, the subgraph $G[C]$ is connected.
\end{lemma}

\begin{proof}
By construction $E(G)=E(M)$ and, since $V$ is a basis of $M$,
for each component $H$ of $G$ we have $r(E(H))=|V(H)|$.
Moreover, for each vertex $v$ of $G$, the hyperplane
of $M$ spanned by $V-\{v\}$ is $E(G-v)$.
Hence $G$ is a weak framework for $M$.
Finally, if $H$ is a subgraph of $G$, then
$r_M(E(H))$ is the sum, taken over all components
$H'$ of $H$, of $r_M(E(H'))$. Therefore,
if $C$ is a circuit of $M$, then $G[C]$ is connected.
\end{proof}

We can now prove that
every frame matroid is quasi-graphic.
\begin{proof}[Proof of Theorem~\ref{main}.]
Let $M$ be a frame matroid.  
Recall that the class of quasi-graphic matroids is closed under taking 
minors, so we may assume that $M$ has a basis
$V$ such that $(M,V)$ is a framed matroid.
Moreover, $M$ is quasi-graphic if and only if its simplification is,
so we may assume that $M$ is simple. 
Now it follows from Lemma~\ref{frame1} that $M$ is
quasi-graphic.
\end{proof}

Next we characterize
frame matroids using frameworks. These results are due
to Zaslavsky~[\ref{zaslavsky}, \ref{zaslavsky2}] but we include
proofs for completeness since they play a central
role in this paper.

Let $G$ be a graph and let $\cB$ be a subset of the cycles
of $G$. We say that $\cB$ satisfies the {\em theta-property}
if there is no theta in $G$ with exactly two of its three cycles
in $\cB$. The following result is contained in [\ref{zaslavsky}, Theorem 2.1].
\begin{theorem}\label{zas1}
Let $G$ be a graph and let $\cB$ be a collection of cycles
in $G$ that satisfy the theta-property.  Now let
$\cI$ denote the collection of all sets $I\subseteq E(G)$ such that
there is no $C\in \cB$ with
$E(C)\subseteq I$ and $|E(H)|\le |V(H)|$ for each component
$H$ of $G[I]$.
Then $\cI$ is the collection of independent sets of a matroid with 
ground set $E(G)$.
\end{theorem}

\begin{proof}
To prove that $(E(G),\cI)$ is a matroid it suffices to check
the following conditions, which are effectively a reformulation
of the circuit axioms in terms of independent sets:
\begin{itemize}
\item[(a)] $\emptyset\in \cI$,
\item[(b)] for each $J\in\cI$ and $I\subseteq J$, we have $I\in\cI$, and
\item[(c)] for each set $I\in \cI$ and $e\in E(G)-I$ either
$I\cup \{e\}\in \cI$ or there is a unique minimal subset
$C$ of $I\cup\{e\}$ that is not in $\cI$.
\end{itemize}
Conditions $(a)$ and $(b)$ follow from the construction.

We call the cycles of $G$ in $\cB$ {\em balanced}.
Let $I\in \cI$ and $e\in E(G)-I$ with $I\cup\{e\}\not\in \cI$.
Let $C_1$ and $C_2$ be minimal subsets of $I\cup\{e\}$ 
that are not in $\cI$. Suppose for a contradiction that 
$C_1\neq C_2$. By definition, for each $i\in\{1,2\}$,
we have $G[C_i-\{e\}]$ is connected, $e\in C_i$,  and either
$G[C_i]$ is a balanced cycle or $|C_i|>|V(C_i)|$.
Consider $J=(C_1\cup C_2)-\{e\}$.
Since $J\subseteq I$, we have $J\in\cI$.
Since $G[C_1-\{e\}]$ and $G[C_2-\{e\}]$ are connected,
$G[J]$ is connected. Therefore $|J|\le |V(J)|$.
It follows that $|C_1|\le |V(C_1)|$ and
$|C_2|\le |V(C_2)|$. Hence $G[C_1]$ and $G[C_2]$ are
balanced cycles. Now $G[J]$ is the union of two paths, each
connecting the ends of $e$, and $|J|\le |V(J)|$, so
$G[C_1\cup C_2]$ is a theta.
By the theta-property, $G[J]$ has a balanced cycle.
However, this contradicts the fact that $J\in\cI$.
\end{proof}

We denote the matroid $(E(G),\cI)$ in Theorem~\ref{zas1} by $FM(G,\cB)$.
The following result is an easy application of [\ref{zaslavsky}, Theorem 2.1].
\begin{theorem}\label{zas2}
If $G$ is a graph and $\cB$ is a collection of cycles
in $G$ that satisfies the theta-property,
then $FM(G,\cB)$ is a frame matroid.
\end{theorem}

\begin{proof}
Let $G^+$ be obtained from $G$ by adding a loop-edge $e_v$
at each vertex of $v$. 
Since we only added loop-edges, the pair $(G^+,\cB)$ still
satisfies the theta-property. Let $M^+=FM(G^+,\cB^+)$ and
$V=\{e_v\, : \, v\in V(G)\}$. By the definition
of $FM(G^+,\cB^+)$, the set $V$ is a basis of $M^+$.
For each non-loop edge $e$ of $G$ with ends $u$ and $v$,
the set $\{e_u,e,e_v\}$ is a circuit of $M^+$ and
for each loop-edge $e$ of $G$ at $v$,
the set $\{e,e_v\}$ is a circuit of $M^+$.
Therefore $M^+$ is a framed matroid and hence
$FM(G,\cB)$ is a frame matroid.
\end{proof}

The next result follows directly from Lemma~\ref{circuit}
and the definition of $FM(G,\cB)$.
\begin{lemma}\label{frame2}
Let $G$ be a framework for a matroid $M$ and let
$\cB$ denote the set of non-balanced cycles of $(M,G)$.
Then $M=FM(G,\cB)$ if and only if for each circuit 
$C$ of $M$ the subgraph $G[C]$ is connected.
\end{lemma}

The following result is the main theorem in~[\ref{zaslavsky2}].
\begin{theorem}\label{zas3}
A matroid $M$ is a frame matroid if and only if there
is a graph $G$ and a collection $\cB$ of cycles of $G$
satisfying the theta-property such that $M=FM(G,\cB)$.
\end{theorem}

\begin{proof}
The ``if" direction of the result follows from 
Theorem~\ref{zas2}. For the converse,
since it is straightforward to
add loops and parallel elements, we may assume  that $M$ is simple, and
that  $(M,V)$ is a framed matroid for some basis $V$ of $M$.
Let $G$ be the support graph of $(M,V)$
and let $\cB$ denote the set of cycles $C$ of $G$
such that $E(C)$ is dependent in $M$.
By Lemma~\ref{frame1},
$G$ is a framework for $M$ and, for each circuit $C'$ of $M$,
the subgraph $G[C']$ is connected.
Now, by Lemma~\ref{frame2}, we have $M=FM(G,\cB)$.
\end{proof}

\section{Lifted-graphic matroids}
\label{sec:lifts}

We say that a matroid $M$ is a {\em lift} of a matroid $N$
if there is a matroid $M'$ and an element $e\in E(M')$
such that $M'\del e = M$ and $M'\con e = N$.
The following result implies Theorem~\ref{lifts}.

\begin{theorem}\label{lifts2}
If $G$ is a graph and $M$ is a lift of $M(G)$, then
$G$ is a framework for $M$. Moreover, if $C_1$ and $C_2$
are disjoint cycles in $G$, then $E(C_1\cup C_2)$ is
dependent in $M$.
\end{theorem}

\begin{proof}
Let $e$ be an element of a matroid $M'$ such that
$M'\del e=M$ and $M'\con e = M(G)$.
Thus $E(M)=E(G)$. For each component $H$ of $G$,
$r_{M'\con e}(E(H)) = |V(H)|-1$ so
$r_M(E(H))= r_{M'}(E(H))\le r_{M'\con e}(E(H))+1 = |V(H)|$.
For a vertex $v$ of $G$, we have 
$\cl_M(E(G-v)) \subseteq
\cl_{M'}(E(G-v)\cup\{e\})-\{e\} = \cl_{M'\con e}(E(G-v))
\subseteq E(G-v)\cup \loops_G(v)$.
So $G$ is a weak framework for $M$.

Now consider two disjoint cycles $C_1$ and $C_2$ of $G$
and let $X=G[C_1\cup C_2]$.
So 
$r_{M}(X) = r_{M'}(X) \le r_{M'\con e}(X) + 1
= r_{M(G)}(X) +1 = |X| -1$ and, hence,
$X$ is dependent. Condition (4) follows.
\end{proof}

Next we will give an alternate characterization of 
lifted-graphic matroids using frameworks;  again, these results
are due to Zaslavsky~[\ref{zaslavsky}, \ref{zaslavsky3}], but the proofs
are included here for completeness. 

\begin{theorem}\label{zas4}
Let $G$ be a graph and let $\cB$ be a collection of cycles
in $G$ that satisfy the theta-property.  Now let
$\cI$ denote the collection of all sets $I\subseteq E(G)$ such that
there is no $C\in \cB$ with
$E(C)\subseteq I$ and $G[I]$ contains at most one cycle.
Then $\cI$ is the set of independent sets of a matroid on $E(G)$.
\end{theorem}

\begin{proof}
As noted in the proof of Theorem~\ref{zas1},
to prove that $(E(G),\cI)$ is a matroid it suffices to check
the following conditions:
\begin{itemize}
\item[(a)] $\emptyset\in \cI$,
\item[(b)] for each $J\in\cI$ and $I\subseteq J$, we have $I\in\cI$, and
\item[(c)] for each set $I\in \cI$ and $e\in E(G)-I$ either
$I\cup \{e\}\in \cI$ or there is a unique minimal subset
$C$ of $I\cup\{e\}$ that is not in $\cI$.
\end{itemize}
Conditions $(a)$ and $(b)$ follow from the construction.

We call cycles of $G$ in $\cB$ {\em balanced}.
Let $I\in \cI$ and $e\in E(G)-I$ with $I\cup\{e\}\not\in \cI$.
Let $C_1$ and $C_2$ be minimal subsets of $I\cup\{e\}$ 
that are not in $\cI$. Suppose for a contradiction that 
$C_1\neq C_2$. By definition,  for each $i\in\{1,2\}$, either
$G[C_i]$ is a balanced cycle,
$G[C_i]$ is the union of two vertex disjoint non-balanced cycles,
or $G[C_i]$ is $2$-edge-connected and $|C_i|=|V(C_i)|+1$.
Consider $J=(C_1\cup C_2)-\{e\}$.
Since $J\subseteq I$, we have $J\in\cI$ so either
$G[J]$ is a forest or $G[J]$ contains a unique cycle.

For each $i\in\{1,2\}$, there is a cycle $A_i$ 
of $G[C_i]$ that contains $e$.
Since $G[J]$ contains at most one cycle,
either $A_1=A_2$ or $A_1\cup A_2$ is a theta.

First suppose that $A_1=A_2$. Since $C_1\neq C_2$, the cycle
$A_1$ is non-balanced. Therefore, for each $i\in\{1,2\}$,
there is a non-balanced cycle $B_i$ in $G[C_i-e]$.
Since $G[J]$ contains a unique cycle $B_1=B_2$.
But then $C_1= E(A_1\cup B_1)$ and $C_2=E(A_2\cup B_2)$,
contradicting the fact that $C_1\neq C_2$.

Now suppose that $A_1\cup A_2$ is a theta, and let $C$ be the
cycle in $(A_1\cup A_2)-e$. Since $J$ is independent,
$C$ is not balanced. By the theta-property and symmetry,
we may assume that $A_1$ is not balanced. Then
there is a non-balanced cycle $B_1$ in $G[C_1-\{e\}]$.
Since $G[J]$ has at most one cycle $C=B_1$.
Therefore $C_1= E(A_1\cup A_2)$ and, hence, 
$A_2$ is non-balanced.  Then
there is a non-balanced cycle $B_2$ in $G[C_2-\{e\}]$.
Since $G[J]$ has at most one cycle $C=B_2$, however,
this contradicts the fact that $C_1\neq C_2$.
\end{proof}

We denote the matroid $(E(G),\cI)$ in Theorem~\ref{zas4}
by $LM(G,\cB)$.
\begin{theorem}\label{zas5}
If $G$ is a graph and $\cB$ is a collection of cycles
in $G$ that satisfies the theta-property,
then $LM(G,\cB)$ is a lift of $M(G)$ and, hence,
$G$ is a framework of $LM(G,\cB)$.
\end{theorem}

\begin{proof}
Let $G^+$ be obtained from $G$ by adding a loop-edge $e$
at a vertex $v$.
Note that $(G^+,\cB)$ satisfies the theta-property;
let $M^+=LM(G^+,\cB)$.
By the definition
of $LM(G^+,\cB)$, for each cycle $C$ of $G$,
$\{e\}\cup E(C)$ is dependent in $M^+$.
Hence $E(C)$ is a dependent set $M^+\con e$.
Similarly,
by the definition
of $LM(G^+,\cB)$, for each forest $F$ of $G$, the set
$\{e\}\cup E(F)$ is independent in $M^+$ and, hence,
$E(F)$ is independent in $M^+\con e$.
Thus $M^+\con e=M(G)$ and, hence,
$M$ is a lift of $M(G)$. So, by
Theorem~\ref{lifts}, $G$ is a framework for $LM(G,\cB)$.
\end{proof}

The following result is a direct consequence of
Lemma~\ref{circuit} and the definition of $LM(G,\cB)$.
\begin{lemma}\label{lifts3}
Let $G$ be a framework for a matroid $M$ and
let $\cB$ denote the set of balanced cycles of $(M,G)$.
Then $M=LM(G,\cB)$ if and only if for each pair $(C_1,C_2)$
of disjoint cycles of $G$, the set $E(C_1\cup C_2)$ 
is dependent in $M$.
\end{lemma}

The following result, which is a converse to Theorem~\ref{zas5},
is proved in [\ref{zaslavsky3}, Section 3].
\begin{theorem}\label{zas6}
If $G$ is a graph, $M$ is a lift of $M(G)$, and
$\cB$ is the set of balanced cycles of $(M,G)$,
then $M=LM(G,\cB)$.
\end{theorem}

\begin{proof}
The result follows immediately from 
Lemmas~\ref{lifts2} and~\ref{lifts3}.
\end{proof}

\section{Frameworks with a loop-edge}

In this section we prove Theorem~\ref{sufficient} which is an
immediate consequence of the following two results.
\begin{theorem}\label{sufficient1}
Let $G$ be a framework for a $3$-connected matroid $M$,
let $\cB$ be the set of balanced cycles of $G$,
and let $e$ be a non-balanced loop-edge at a vertex $v$.
If $e\in\cl_{M}(E(G-v))$, then $M=LM(G,\cB)$.
\end{theorem}

\begin{proof}
By Lemmas~\ref{conn3} and~\ref{xx}, we may assume that $G$ is connected.

We will start by proving, for any non-balanced 
cycle $C$ of $G$, that $E(C)\cup\{e\}$ is a circuit of $M$.
By Lemmas~\ref{dependentset} and~\ref{circuit} we may assume that
$v\not \in V(C)$.
Let $P$ be a minimal path from $\{v\}$ to $V(C_2)$ and let 
$X=\{e\}\cup E(P\cup C)$.  By Lemma~\ref{dependentset}, $X$
is dependent.
Let $f$ be the edge of $P$ that is incident with $v$.
By $(3)$ and the fact that $e\in\cl_{M}(E(G-v))$,
there is a cocircuit $C^*$ of $M$ such that
$C^*\cap X =\{f\}$.
Therefore $X$ is not a circuit of $M$.
So, by Lemma~\ref{circuit}, $\{e\}\cup E(C)$ is a 
circuit of $M$, as required.

By Lemma~\ref{lifts3}, 
it suffices to prove that if $C_1$ and $C_2$ are vertex-disjoint
cycles of $G$, then $E(C_1\cup C_2)$ is dependent in $M$;
we may assume that $C_1$ and $C_2$ are non-balanced.
By the preceding paragraph
we may assume that neither $C_1$ nor $C_2$ is equal to
$G[\{e\}]$ and both
$E(C_1)\cup \{e\}$ and $E(C_2)\cup \{e\}$ are 
circuits of $M$. So, by the circuit-exchange property,
$E(C_1\cup C_2)$ is dependent, as required.
\end{proof}

\begin{theorem}\label{sufficient2}
Let $G$ be a framework for a $3$-connected matroid $M$ with $|M|\ge 4$,
let $\cB$ be the set of balanced cycles of $G$, and let
$e$ be a loop-edge at a vertex $v$.
If $e\not\in\cl_{M}(E(G-v))$, then $M=FM(G,\cB)$.
\end{theorem}

\begin{proof}
First we consider the case that $G$ is not connected.
By Lemma~\ref{conn3}, $G$ has exactly two components
one of which is a loop-component; let $f$ be the edge
in the loop-component. By Lemma~\ref{xx},
$M\con f = M(G\del f)$. Since $M$ is $3$-connected,
$M\con f$ has no loops and, hence, $f=e$.
However, $e\not\in\cl_{M}(E(G-v))$ and hence $e$ is a coloop of $M$,
contradicting fact that $M$ is $3$-connected. Hence $G$ is connected.
 
Suppose by way of contradiction that $M\neq FM(G,\cB)$.
Then, by (4) and Lemmas~\ref{circuit} and~\ref{frame2},
there exist disjoint
non-balanced cycles $C_1$ and $C_2$ of $(M,G)$
such that $E(C_1\cup C_2)$ is a circuit in $M$.

Since $e\not\in \cl_{M}(E(G-v))$, 
neither $C_1$ nor $C_2$ is equal to $G[\{e\}]$.
Since $G$ is connected, there is a path
from $v$ to $V(C_1\cup C_2)$ in $G$;
let $P$ be a minimal such path.
We may assume that $P$ has an end in $V(C_1)$.
By Lemmas~\ref{dependentset} and~\ref{circuit},
$E(C_1\cup P)\cup \{e\}$ is a circuit of $M$.
Let $f\in E(C_1)$; by the circuit exchange property, there 
exists a circuit $C$ in $(E(C_1\cup C_2\cup P)\cup \{e\})-\{f\}$.
By Lemma~\ref{circuit}, $C= E(C_2)\cup\{e\}$.
However this contradicts the fact that $e\not\in\cl_{M}(E(G-v))$.
\end{proof}

\section{Representable matroids}

A framework $G$ for a matroid $M$ is called {\em strong}
if $G$ is connected and $r_M(E(G-v)) = r(M)-1$
for each vertex $v$ of $G$.
\begin{lemma}\label{strong}
If $M$ is a $3$-connected quasi-graphic matroid with $|M|\ge 4$, then
$M$ has a strong framework.
\end{lemma}

\begin{proof}
By Lemma~\ref{conn3}, $M$ has a connected framework.
Let $G$ be a connected framework having as many 
loop-edges as possible.
Suppose that $G$ is not a strong framework and
let $v\in V(G)$ such that $r_M(E(G)-v) < r(M)-1$.
Let $C^*$ be a cocircuit of $M$ with 
$C^*\cap E(G-v)=\emptyset$; if possible we choose
$C^*$ so that it contains a loop-edge of $G$.
Since $M$ is $3$-connected, $|C^*|\ge 2$ and, by
Lemma~\ref{dependentset}, there is at most one loop-edge
at $v$. Therefore $C^*$ contains at least one non-loop-edge.
Let $L$ denote the set of non-loop-edges of $G-C^*$ incident with
$v$. By our choice of $C^*$, the set $L$ is non-empty.

Let $H$ be the graph obtained from $G$ by replacing each 
edge $f=vw\in L$ with a loop-edge at $w$. By Lemma~\ref{conn4},
$H$ is connected.  Note that $H$ is framework for $M$. However,
this contradicts our choice of $G$.
\end{proof}

We can now prove our main theorem that,
if $M$ is a $3$-connected representable quasi-graphic matroid, then 
$M$ is either a frame matroid or a lifted-graphic matroid.

\begin{proof}[Proof of Theorem~\ref{main}.]
Let $M=M(A)$, where $A$ is a matrix over a field $\bF$
with linearly independent rows.
We may assume that $|M|\ge 4$. Therefore,
by Lemma~\ref{strong},
$M$ has a strong framework $G$.

\begin{claim}
There is a matrix $B\in \bF^{V(G)\times E(G)}$ 
such that 
\begin{itemize}
\item the row-space of $B$ is contained in the
row-space of $A$, and
\item
for each $v\in V(G)$ and non-loop edge
$e$ of $G$, we have $B[v,e]\neq 0$ if and only if
$v$ is incident with $e$.
\end{itemize}
\end{claim}

\begin{proof}[Proof of claim.]
Let $v\in V(G)$ and let $C^* = E(M)-\cl_M(E(G-v))$.
By the definition of a strong framework,
$C^*$ is a cocircuit of $M$. 
Since $r(E(M)-C^*)<r(M)$, by applying row-operations
to $A$ we may assume that there is a row $w$ of $A$ whose support
is contained in $C^*$. Since $C^*$ is minimally co-dependent,
the support of row-$w$ is equal to $C^*$. Now
we set the row-$v$ of $B$ equal to the row-$w$ of $A$.
\end{proof} 

Note that $M(B)$ is a frame matroid and $G$ is a 
framework for $M(B)$. We may assume that 
$r(M(A))>r(M(B))$ since otherwise $M(A)$ is a frame matroid.
Since $G$ is a connected framework for both $M(A)$ and $M(B)$,
it follows that $r(M(B)) = |V(G)|-1$ and that 
$r(M(A)) = |V(G)|$. Up to row-operations we may assume that
$A$ is obtained from $B$ by appending a single row.
By Lemma~\ref{graphic}, $M(B)= M(G)$.
Hence $M$ is a lift of $M(G)$.
\end{proof}

\section*{Acknowledgement}
We thank Daryl Funk for pointing out a significant error in an 
earlier version of this paper. We also thank the anonymous
referees for their detailed comments.

\section*{References}

\newcounter{refs}

\begin{list}{[\arabic{refs}]}%
{\usecounter{refs}\setlength{\leftmargin}{10mm}\setlength{\itemsep}{0mm}}

\item \label{oxley}
J. G. Oxley,  {\em Matroid Theory}
Oxford University Press, New York, second edition (2011).

\item \label{rw}
R. Chen, G. Whittle,
On recognizing frame and lifted graphic matroids,
to appear in J. Graph Theory.

\item\label{seymour}
P. D. Seymour,
Recognizing graphic matroids,
Combinatorica {\bf 1} (1981), 75--78.

\item \label{zaslavsky}
T. Zaslavsky,
Biased graphs. II. The three matroids,
J. Combin. Theory Ser. B {\bf 47} (1989),  32-52.

\item \label{zaslavsky2}
T. Zaslavsky,
Frame matroids and biased graphs,
Europ. J. Combinatorics {\bf 15} (1994),  303-307.

\item \label{zaslavsky3}
T. Zaslavsky,
Supersolvable frame-matroids and graphic-lift lattices,
Europ. J. Combinatorics {\bf 15} (2001),  119-133.

\end{list}

\end{document}